\documentclass[11pt,reqno]{amsart}
\usepackage{tikz,hyperref}
\usepackage[backend=bibtex, style=alphabetic]{biblatex}
\usepackage{amsfonts,amsmath,amssymb,amsthm}
\usepackage{xcolor}
\usepackage[left=3cm,right=3cm,top=3cm,bottom=3cm]{geometry}
\usepackage{palatino}
\newtheorem{theorem}{Theorem}
\newtheorem{remark}{Remark}
\newtheorem{proposition}{Proposition}
\newtheorem{corollary}{Corollary}
\newtheorem{lemma}{Lemma}

\newtheorem{example}{Example}
\newcommand\bremark{\begin{remark}\begin{upshape}}
\newcommand\eremark{\end{upshape}\end{remark}}
\allowdisplaybreaks
\title{Flagged Skew Schur Polynomials Twisted By Roots Of Unity}
\author[]{V. Sathish Kumar}
\address{The Institute of Mathematical Sciences, A CI of Homi Bhabha National Institute, Chennai 600113, India}
\email{vsathish@imsc.res.in}
\date{}
\thanks{The author acknowledges partial funding from a DAE Apex Project grant to the Institute of Mathematical Sciences, Chennai.}
\begin{document}
\begin{abstract}
    We generalize a theorem of Littlewood concerning the factorization of Schur polynomials when their variables are twisted by roots of unity. We show that a certain family of flagged skew Schur polynomials admit a similar factorization. These include an interesting family of Demazure characters as a special case.
\end{abstract}
\maketitle
\section{Preliminaries and Background}
A \emph{partition} $\lambda = (\lambda_1, \lambda_2, \cdots)$ is a weakly decreasing sequence of non-negative integers with finitely many non-zero terms (or parts). The \emph{length} of $\lambda$, denoted by $l(\lambda)$ is the largest integer $i$ such that $\lambda_i$ is non-zero.
The \emph{Young diagram} of $\lambda$ is the left and top justified collection of boxes such that the $i^{th}$ row contains $\lambda_i$ boxes. By abuse of notation we denote the Young diagram of $\lambda$ again by $\lambda$. For Young diagrams $\lambda$ and $\mu$ such that $\mu \subset \lambda $ (i.e., $\mu_i \leq \lambda_i \ \text{ } \forall i $), the \emph{skew diagram} $\lambda / \mu$ is obtained by removing the boxes of $\mu$ from that of $\lambda$. (See Figure~\ref{skew tab}).

Given a positive integer $t$, there are $t+1$ partitions associated to $\lambda$ namely: $core_t(\lambda)$, $\lambda^{(0)}$, $\lambda^{(1)}$, $\cdots$, $ \lambda^{(t-1)}$. The partition $core_t(\lambda)$ is the \emph{t-core} of $\lambda$ and the tuple of partitions $\left(\lambda^{(0)}, \lambda^{(1)}, \cdots, \lambda^{(t-1)}\right)$ is the \emph{t-quotient} of $\lambda$. Refer~\cite[Chapter I, \S 1, Example-8]{macdonald} for details.

A \emph{semi-standard skew tableau} of shape $\lambda / \mu$ is a filling of the skew diagram $\lambda / \mu$ that is weakly increasing along the rows (from left to right) and strictly increasing along the columns (from top to bottom). We denote by Tab$(\lambda / \mu)$ the set of all semi-standard skew tableaux of shape $\lambda / \mu$. The \emph{weight} of $T \in$ Tab$(\lambda /\mu)$ is defined as $wt(T) := (t_1, t_2, \cdots)$, where $t_i$ is the number of times $i$ occurs in $T$.

A \emph{flag} $a$ = $(a_1, a_2, \cdots)$ is a weakly increasing sequence of non-negative integers (we will usually write only up to a first few terms that are relevant in the context for the sake of brevity). Let $\lambda / \mu$ be a skew diagram and $a, b$ be flags. The set of all flagged skew tableaux of shape $\lambda / \mu$ with flags $a, b$ is defined by $$\text{Tab}(\lambda /\mu;a,b) := \{T \in \text{Tab}(\lambda /\mu) \;| \; a_i < T_{ij}\leq b_i, \quad \forall i,j\}$$
where, $T_{ij}$ is the entry of the box in the $i^{th}$ row and $j^{th}$ column of $T$.
\begin{figure}
    \centering
    \begin{tikzpicture}[scale=0.7]
        \draw (0,0) -- (0,3);
        \draw (5,5) -- (5,4) -- (1,4);
        \draw (0,0) -- (1,0) -- (1,4);
        \draw (0,1) -- (2,1) -- (2,5);
        \draw (5,5) -- (2,5);
        \draw (0,2) -- (3,2) -- (3,5);
        \draw (0,3) -- (3,3);
        \draw (4,4) -- (4,5);
        \draw (0,0) -- (0,-1) -- (1,-1) -- (1,0);
        \node at (0.5, -0.5) {$6$};
        \node at (0.5, 0.5) {$5$};
        \node at (0.5, 1.5) {$4$};
        \node at (0.5, 2.5) {$3$};
        \node at (1.5, 1.5) {$5$};
        \node at (1.5, 2.5) {$4$};
        \node at (1.5, 3.5) {$2$};
        \node at (2.5, 2.5) {$4$};
        \node at (2.5, 3.5) {$3$};
        \node at (2.5, 4.5) {$1$};
        \node at (3.5, 4.5) {$2$};
        \node at (4.5, 4.5) {$2$};
        \end{tikzpicture}
    \caption{A skew tableau of skew shape $(5, 3, 3, 2, 1, 1) / (2,1)$ with flags $(0, 1, 2, 3,4,4)$ and $(2, 3, 5, 5, 5, 6)$ and weight $(1, 3, 2, 3, 2, 1)$.}
    \label{skew tab}
\end{figure}
The flagged skew Schur polynomial $s_{\lambda / \mu}(a,b)$ is defined by $$s_{\lambda / \mu}(a,b) := \sum _{T \in Tab(\lambda /\mu;a,b)} x^{wt(T)}$$
\bremark
    Flagged skew Schur polynomials  are non-symmetric generalizations of skew Schur polynomials. There are also coincidences between these and several other families of polynomials that we study in algebraic combinatorics namely, \emph{key polynomials} corresponding to $312$-avoiding permutations, \emph{Schubert polynomials} corresponding to vexillary permutations, etc. See~\cite[\S 7]{RS}
\eremark
The following theorem gives a determinantal formula for flagged skew Schur polynomials.
\begin{theorem}
\begin{upshape}\cite[Theorem-3]{Gessel} \cite{Walchs}\end{upshape} Let $\lambda / \mu$ be a skew shape where $\lambda$ has at most n parts. Suppose $a, b$ are flags. Then
\begin{equation}\label{jacobi-trudi}
    s_{\lambda /\mu}(a,b) = det\left(h_{\lambda_i - \mu_j -i +j}(a_j, b_i)\right)_{1\leq i, j \leq n}
\end{equation}
where $h_k(u,v) := h_k(x_{u+1}, x_{u+2}, \cdots, x_{v})$ is the complete symmetric polynomial corresponding to the partition $(k)$ in the variables $\{x_{u+1}, x_{u+2}, \cdots, x_{v}\}$ if $u<v$ and $k>0$. If $k=0$ then $h_{k}(u,v) := 1$ and if $k<0$ then $h_{k}(u,v) := 0$. If $u \geq v$ and $k \neq 0$ then $h_k(u,v) := 0$.
\end{theorem}
\begin{remark}\begin{upshape}
    If we define $s_{\lambda / \mu}(a,b) :=0$ for partitions $\mu \not \subset\lambda$, then still \eqref{jacobi-trudi} holds.
    \end{upshape}
\end{remark}
Fix positive integers $n$ and $t$ throughout the article. For a partition $\lambda$ with at most $tn$ parts, we define the associated beta sequence by 
\begin{equation*}
    \beta(\lambda) := \left(\lambda_1 +tn-1, \lambda_2 + tn-2, \cdots, \lambda_{tn} +0\right) 
\end{equation*}
For $0\leq r \leq t-1$, let $m_r(\lambda)$ denote the number of parts of $\beta(\lambda)$ that are congruent to $r$ modulo $t$. 
Consider the unique permutation $\sigma_{\lambda} \in \mathfrak S _{tn}$, that regroups the parts of $\beta (\lambda)$ according to the residue classes in ascending order while maintaining the descending order within each group. More precisely,
\begin{equation*}
    \beta(\lambda)_{\sigma _{\lambda} (j)} \equiv r \quad (mod \; t),  \quad  \quad \sum_{i=1} ^{r-1} m_r(\lambda) +1 \leq j \leq \sum_{i=1} ^{r} m_r(\lambda)
\end{equation*}
and $\beta(\lambda)_{\sigma _{\lambda} (j)} > \beta(\lambda)_{\sigma _{\lambda} (k)}$ whenever $\sum_{i=1} ^{r-1} m_r(\lambda) +1 \leq j<k \leq \sum_{i=1} ^{r} m_r(\lambda)$. 

Given a skew shape $\lambda /\mu$ with $l(\lambda) \leq tn$, we define $$\epsilon_t(\lambda / \mu):= \delta_{core_t(\lambda), core_t(\mu)} sgn (\sigma_{\lambda}). sgn (\sigma _{\mu})$$ where, $\delta$ is the Kronecker delta.
\section{Previous results}
Let $R[[x_1, x_2, \cdots]]$ denote the ring of formal power series in the variables $x_1, x_2, \cdots$ with coefficients in a ring $R$. We also consider the set of variables $\{y_1, y_2, \cdots\}$ satisfying $y_i ^t = x_i$. One can identify $R[[x_1, x_2, \cdots]]$ inside $R [[y_1, y_2, \cdots]]$ as $ R[[y_1 ^t, y_2 ^t, \cdots]]$. Let $\Lambda \subset \mathbb{C}[[x_1, x_2, \cdots]]$ denote the ring of symmetric functions with integral coefficients in the variables $\{x_1, x_2, \cdots\}$.  

Let $\omega$ be a primitive $t^{th}$ root of unity. Define maps
$$\phi_t, \psi_t: \mathbb{C}[[x_1, x_2, \cdots]] \longrightarrow \mathbb{C}[[y_1, y_2, \cdots]]$$
by $\phi_t f(x_1, x_2, \cdots) := f(y_1, \omega y_1, \omega ^2 y_1, \cdots, \omega ^{t-1} y_1, \hspace{0.2cm} y_2, \omega y_2, \omega ^2 y_2, \cdots, \omega ^{t-1} y_2, \cdots)$\\ and $\psi_t f(x_1, x_2, \cdots) := f (x_1 ^t, x_2 ^t, \cdots)$.
\bremark
The operator $\psi_t$ when restricted to $\Lambda$ is the plethysm by power-sum symmetric function $p_t$ and $\phi_t$ is the adjoint of $\psi_t$ with respect to the Hall inner product on $\Lambda$.
\eremark
It is well known that
\begin{equation} \label{phi_action} \phi_t\left(h_k (x_1, x_2, \cdots)\right) =  \left\{ 
        \begin{array}{ll}
        h_{k/t}(y_1 ^t, y_2 ^t, \cdots) = h_{k/t}(x_1, x_2, \cdots) & \quad \text{if} \quad t | k  \\
        0 & \quad \text{otherwise}
        \end{array}    \right.
\end{equation} where $h_k$, for $k \geq 0$, is the complete homogeneous symmetric function corresponding to the partition $(k)$. See \cite[Chapter I, \S 5, Example 24 (a)]{macdonald}
\begin{theorem}\label{littlewood_theorem}
    \begin{upshape}\cite[Equation 7.3{;}3]{Little} \cite[Theorem 2]{Dipendra}\end{upshape} Let $\lambda$ be a partition with at most $tn$ parts. Then,
    \begin{equation} \label{little eqn}
        s_{\lambda}\left((X_n)^{1/t}\right) = \epsilon_t(\lambda)\cdot \prod_{k=0} ^{t-1} s_{\lambda^{(k)}}(X_n)
    \end{equation} 
    where $X_n := (x_1, x_2, \cdots, x_n)$  and $(X_n)^{1/t} := \left(y_1, \omega y_1, \cdots, \omega ^{t-1} y_1, \cdots, y_n, \omega y_n, \cdots, \omega ^{t-1} y_n \right)$. 
\end{theorem}
One can re-write~\eqref{little eqn} as $\phi_t \left(s_{\lambda}(x_1, \cdots, x_{tn})\right) = \epsilon_t(\lambda)\cdot \prod_{k=0} ^{t} s_{\lambda^{(k)}}(x_1, x_2, \cdots, x_n)$.\\ Here, $s_{\lambda}(x_1, \cdots, x_{tn})$ denotes the Schur polynomial corresponding to the partition $\lambda$ in the variables $\{x_1, x_2, \cdots, x_{tn}\}$.
\begin{theorem}
    \begin{upshape}\cite[Chapter I, \S 5, Example 24]{macdonald} \end{upshape} \label{macdonald_theorem} 
        If $\mu \subset \lambda$, then
        $$\phi_t(s_{\lambda / \mu}) = \epsilon_t(\lambda /\mu) \prod _{r=0} ^{t-1} s_{\lambda ^{(r)} / \mu^ {(r)}}$$    
\end{theorem}
    Here, $s_{\lambda /\mu}$ is the skew Schur function corresponding to the skew shape $\lambda /\mu$. Note that Theorem~\ref{macdonald_theorem} is a generalization of Theorem~\ref{littlewood_theorem} at the level of symmetric functions (infinitely many variables). Theorem~\ref{littlewood_theorem} is further generalized to other \emph{Cartan types} and to the level of \emph{universal characters} in~\cite{nishu1} and \cite{Albion} respectively.
\bremark
    Observe that since $\{s_{\lambda}(x_1, \cdots, x_{tn})\;|\; \; l(\lambda) \leq tn\}$ forms a basis of $\Lambda_{tn}$ (the ring of all symmetric polynomials with integral coefficients in $tn$ variables $\{x_1, \cdots, x_{tn}\}$), Theorem~\ref{littlewood_theorem} implies (among other things) that the image of $\Lambda_{tn}$ is $\Lambda_{n}$.
\eremark
\section{Main Theorem}
Our main theorem describes the action of $\phi_t$ on flagged skew Schur polynomials.
\begin{lemma}
    Let $u, v \in t\mathbb{Z}$ and $k \in \mathbb Z$. Then $$\phi_t\left(h_{k}(u, v)\right) = \left \{ \begin{array}{cc}
        h_{k/t}\left(u/t, v/t\right) & \text{if} \; \;k \in t\mathbb Z \\
        0 & \text{otherwise}
    \end{array}\right.$$
\end{lemma}
\begin{proof}
    Follows from \eqref{phi_action}.
\end{proof}
\begin{theorem}\label{main_theorem}
    Let $\lambda / \mu$ be a skew shape and $a$, $b$ be flags such that $a_i, b_i \in t \mathbb{Z}$. Then,
    $$\phi_t(s_{\lambda / \mu}(a,b)) = \epsilon_t(\lambda / \mu) \prod_{r=0} ^{t-1} s_{\lambda^{(r)}/\mu ^{(r)}}\left(a^{(r)}, b^{(r)}\right)$$
    where, $t \cdot a^{(r)} = \left(a_{\sigma_{\mu}(m_0 + \cdots + m_{r-1}+1)}, \cdots,  a_{\sigma_{\mu}(m_0 + \cdots + m_{r-1}+ m_r)}\right)$ and \\ $t \cdot b^{(r)} = \left(b_{\sigma_{\lambda}(m_0 + \cdots + m_{r-1}+1)}, \cdots,  b_{\sigma_{\lambda}(m_0 + \cdots +m_{r-1}+ m_{r})}\right)$
\end{theorem}
\begin{proof}
Consider the Jacobi-Trudi determinant in~\eqref{jacobi-trudi}. Rearrange the rows and columns in such a way that the residue classes of $\lambda_i - i$ and $\mu_j - j$ are grouped in ascending order and the values in each group are in descending order. This introduces the sign, $sgn (\sigma_{\lambda}). sgn (\sigma _{\mu})$. 

Now, apply the map $\phi_t$ to the determinant. Then we end up in computing the determinant a diagonal sum of matrices of order $m_r(\lambda) \times m_r(\mu)$ for $0 \leq r \leq t-1$, where $m_r(\lambda)$ (resp. $m_r(\mu)$) is the cardinality of $\{i \hspace{0.2cm}| \hspace{0.2cm} \lambda_i - i \equiv r \hspace{0.2cm}(\text{mod }t)\}$ (resp. $\mu_i - i$). This is because $\phi_t \left(h_{\lambda_i - i - \mu_j +j}(a_j, b_i)\right)$ is non-zero only when $\lambda_i - i \equiv \mu_j - j$ (mod $t$). Since the determinant of a diagonal sum of matrices is non-zero only if the blocks are square matrices, we may therefore conclude that $\phi_t\left(s_{\lambda / \mu}(a,b)\right)= 0$ unless $m_r(\lambda) = m_r(\mu)$ for all $0 \leq r \leq t-1$. This amounts for the $\delta_{core_t(\lambda), core_t(\mu)}$ in $\epsilon_t(\lambda /\mu)$. 

Let us now assume that $m_r(\lambda) = m_r(\mu) =: m_r$ for all $0 \leq r \leq t-1$. The entries of the minor corresponding to the residue class $r$ are given by:
\begin{equation*}
    \phi_t\left(h_{\lambda_i - i -\mu_j +j}(a_j, b_i)\right) = h_{\lambda_k ^{(r)} - \mu_l ^{(r)} -k +l}\left(\frac{a_j}{t},\frac{b_i}{t}\right)
\end{equation*}
for some $k$ and $l$ . This is because, $(\lambda_i + tn - i -r)/t = \lambda_k ^{(r)} + m_r - k$ if $\beta(\lambda) _i$ is the $k^{th}$ largest element among the coordinates of $\beta(\lambda)$ that are congruent to $r$ modulo $t$. Similarly, $(\mu_j + tn - j -r)/t = \mu_l ^{(r)} + m_r - l$ if $\beta(\mu)_j$ is the $l^{th}$ largest among the coordinates of $\beta(\mu)$ that are congruent to $r$ modulo $t$. Therefore by \eqref{jacobi-trudi}, the minor corresponding to the residue class $r$ evaluates to $s_{\lambda^{(r)}/ \mu^{(r)}} (a^{(r)}, b^{(r)})$
where $a^{(r)} = \frac{1}{t}\left(a_{\sigma_{\mu}(\sum _{i=0} ^{r-1} m_{i}+1)}, \cdots,  a_{\sigma_{\mu}(\sum _{i=0} ^{r-1} m_{i} + m_r)}\right)$ and\\ $b^{(r)} = \frac{1}{t}\left(b_{\sigma_{\lambda}(\sum _{i=0} ^{r-1} m_{i}+1)}, \cdots,  b_{\sigma_{\lambda}(\sum _{i=0} ^{r-1} m_{i}+ m_{r})}\right)$.

We have therefore shown that if $m_r(\lambda) = m_r(\mu)$
for all $0 \leq r \leq t-1$, then
\begin{equation*}
    \phi_t(s_{\lambda / \mu}(a,b)) =  sgn (\sigma_{\lambda}). sgn (\sigma _{\mu}) \prod _{r=0} ^{t-1} s_{\lambda^{(r)}/\mu^{(r)}}\left(a^{(r)}, b^{(r)}\right)
\end{equation*}
\end{proof}
\bremark
Observe that the RHS in Theorem~\ref{littlewood_theorem} is non-zero if $\epsilon_t(\lambda)$ is non-zero. But this need not be the case with Theorems~\ref{macdonald_theorem} and \ref{main_theorem}, as 
\begin{enumerate}
    \item It does not follow from $core_t(\lambda) = core_t(\mu)$ that $\mu^{(k)} \subset \lambda^{(k)},\quad 0 \leq k \leq t-1$. (Consider for example $\lambda = (5,3,2,0)$ and $\mu=(4,0,0,0)$. In this case the Jacobi-Trudi determinant \eqref{jacobi-trudi} vanishes for any flags $a$ and $b$.)
    \item Flagged skew Schur polynomial may be zero if the flags are not ``compatible" with the corresponding skew shape.
\end{enumerate}
\eremark
\begin{example} 
    \begin{upshape}
        Let $t=2$. Then, for partitions $\lambda = (5,4,4,3)$, $\mu = (0)$ and flags $a = (0,0,0,0)$, $b = (2,2,4,4)$ we have:
        \begin{itemize}
            \item $s_{\lambda / \mu}(a,b) = s_{\lambda}(a,b) = x_1 ^5 x_2 ^4 x_3 ^3 x_4 ^3 (x_3 + x_4)$
            \item $core_2(\lambda) = core_2(\mu) = (0)$
            \item $\lambda ^{(0)} = (3,3)$ and $\lambda ^{(1)} = (1,1)$
            \item $\mu ^{(0)} = \mu ^{(1)} = (0)$
            \item $a^{(0)} = a^{(1)} = (0)$, $b^{(0)} = (1,1)$ and $b^{(1)} = (3,3)$
            \item $\sigma _{\lambda} = [1,2,3,4]$, $\sigma _{\mu} = [2,4,1,3]$ and hence $\epsilon_2 (\lambda / \mu) = -1$
            \item $\phi_2(s_{\lambda}(a,b)) = s_{\lambda}(a,b)(x_1, -x_1, x_2, -x_2) = 0$
            \item $s_{(3,3)} \left(a^{(0)}, b^{(0)}\right) \cdot s_{(1,1)} \left(a^{(1)},b^{(1)}\right) = 0$, because $s_{(3,3)} \left(a^{(0)},b^{(0)}\right) = 0$. 
        \end{itemize}
    \end{upshape}
\end{example}
\begin{remark}
    \begin{upshape}
        There is a column version of flagged skew Schur polynomials \cite[Theorem~3.5*]{Walchs}. A natural analogue for Theorem~\ref{main_theorem} in this case can be obtained by tracing the same steps.
    \end{upshape}
\end{remark}
\section{The Special Case - Demazure Characters}
If $\lambda$ is a partition with at most $n$ parts and $w \in \mathfrak{S}_n$, let $\kappa_{\lambda, w}$ denote the \emph{key polynomial} corresponding to $\lambda$ and $w$. These are known to be characters of \emph{type-A Demazure modules}. 

A permutation $w = [w_1, w_2, \cdots, w_n]$ (in one-line notation) contains the pattern $312$ if there exists $1 \leq i < j < k \leq n$ such that $w_j < w_k < w_i$. A permutation is \emph{$312$-avoiding} if it does not contain the pattern $312$. 

For $1 \leq k \leq n$, we say $k$ is an \emph{ascent} of a permutation $w$ if $w_{k-1} < w_k$. By convention, $1$ is always an ascent for any permutation $w$.

There is a one-to-one correspondence between $\{a = (a_1, a_2, \cdots, a_n)\;|\;\; a_i \leq a_{i+1}, \text{ } n \geq a_i \geq i\}$ and the set of all $312$-avoiding permutations of $[n]$. The correspondence is described in \cite[section 14]{PS} as follows:
\begin{equation}\label{PS_bijection}
a = (a_1, a_2, \cdots, a_n) \text{ corresponds to } w_a = [w_1, w_2, \cdots, w_n]
\end{equation} 
where $w_i$ is described inductively as $w_1 := a_1$ and $w_k := max\{x: x \leq a_k \text{ and } x \neq w_i \text{ , } 1\leq i \leq k-1\}$.  
\begin{theorem}\label{PS-theorem}
    \cite[Theorem 14.1]{PS} Let $\lambda$ be a partition with at most $n$ parts and $a = (a_1, a_2, \cdots, a_n)$ be a flag such that $a_i \geq i$. Then, $$s_{\lambda/(0)}(0,a) = \kappa_{\lambda, w_a}$$
\end{theorem}
\begin{proposition}\label{tflags_and_permutations}
    The set of flags $\{a = (a_1, a_2, \cdots, a_n): a_i \leq a_{i+1},\text{ } n \geq a_i \geq i \text{ and } a_i \in t\mathbb Z\}$ is in one-to-one correspondence with the set of $312$-avoiding permutations of $[n]$ for which if $k$ is an ascent then $w_k$ is a multiple of $t$.
\end{proposition}
\begin{proof}
    The bijection is given by~\eqref{PS_bijection}. The proof follows from the observation that $k$ is an ascent of $w_a$ if and only if $a_{k-1} < a_k$ in which case $w_k = a_k$.
\end{proof}
From Theorems~\ref{main_theorem}, \ref{PS-theorem} and Proposition~\ref{tflags_and_permutations} we have the following corollary.
\begin{corollary}
    Let $\lambda$ be a partition with at most $tn$ parts. If $w \in \mathfrak{S_{tn}}$ is a $312$-avoiding permutation all of whose ascents $k$ are such that $w_k \in t\mathbb Z$, then either $\phi_t(\kappa _{\lambda, w}) = 0$ or there exists permutations $w^{(0)}, w^{(1)}, \cdots, w^{(t-1)} \in \mathfrak{S}_n$ such that
    $$\phi_t(\kappa_{\lambda, w}) = 
    \epsilon_t(\lambda) \prod _{r = 0} ^{t-1} \kappa_{\lambda^{(r)}, w_r}$$
\end{corollary}
The above corollary can be extended to permutations $\sigma_1 \sigma_2 \cdots \sigma_s \in \mathfrak{S}_{tn_1} \times \mathfrak{S} _{tn_2} \times \cdots \times \mathfrak{S}_{tn_s} \subset \mathfrak{S}_{tn}$ where each $\sigma_k$ is a permutation as described in the corollary. See section 5.2 in~\cite{kushwaha2022saturation}.
\section{Enumeration using t-Dyck paths}
The set of all $312$-avoiding permutations in $\mathfrak S _n$ are enumerated by \emph{Dyck paths} of length $2n$. There are $C_n$ many of them, where $C_n$ is the $n^{th}$ Catalan number.

A \emph{$t$-Dyck path} is a lattice path with nodes $\left((x_i, y_i)\right)_{i=0} ^{(t+1)n}$ in $\mathbb N \times \mathbb N$ such that $(x_0, y_0) = (0,0)$, $\left(x_{(t+1)n},y_{(t+1)n}\right)=(tn,tn)$, $x_i \leq y_i$ and $\left(x_{i+1}, y_{i+1}\right) = (x_i,y_i) + (1,0)$ or $(x_i, y_i) + (0,t)$. The $t$-Dyck paths are actually \emph{$(t+1)$-Raney sequences} as in~\cite[361]{Knuth}. The number of $t$-Dyck paths is given by $\binom{(t+1)n}{n}\frac{1}{tn+1}$ which is the $n^{th}$ \emph{Fuss-Catalan number} $C_n ^{t+1}$.

We have the following interesting observation which we intend to record as an aside.
\begin{proposition}\label{t-dyck_paths}
    The set of all $312$-avoiding permutations in $\mathfrak{S}_{tn}$ whose ascents $k$ are such that $w_k$ is a multiple of $t$ is in one-to-one correspondence with the set of all $t$-Dyck paths of length $tn$. 
\end{proposition}
\begin{proof}
    It is easy to see that $t$-Dyck paths are in one to one correspondence with the set of flags in Proposition~\ref{tflags_and_permutations}. Consider an equilateral trianglular grid of side length $tn$. We encode a given flag $(a_1, a_2, \cdots, a_{tn})$ in the triangle by shading the $k^{th}$ column up to a height of $tn-a_k$ from the bottom. Now one obtains the required $t$-Dyck path by rotating the triangle so that the bottom-right vertex aligns with $(0,0)$ and the top vertex aligns with $(tn, tn)$.
    \begin{figure}[h]
        \centering
        \begin{tikzpicture}[scale = 0.5]   
    		\draw (0,0) -- (9,0);
	   	    \draw (0.5, 0.5*1.732) -- (8.5, 0.5*1.732);
		      \draw (1, 1.732) -- (8, 1.732);
		      \draw (1.5, 1.5*1.732) -- (7.5, 1.5*1.732);
		      \draw (2,2*1.732) -- (7, 2*1.732);
            \draw (2.5,2.5*1.732) -- (6.5, 2.5*1.732);
            \draw (3,3*1.732) -- (6, 3*1.732);
            \draw (3.5,3.5*1.732) -- (5.5, 3.5*1.732);
            \draw (4,4*1.732) -- (5, 4*1.732);
            \draw (0,0) -- (4.5, 4.5*1.732);
		      \draw (1,0) -- (5, 4*1.732);
		      \draw (2,0) -- (5.5, 3.5*1.732);
		      \draw (3,0) -- (6, 3*1.732);
		      \draw (4,0) -- (6.5, 2.5* 1.732);
		      \draw (5,0) -- (7, 2 *1.732);
            \draw (6,0) -- (7.5, 1.5*1.732);
            \draw (7,0) -- (8, 1*1.732);
            \draw (8,0) -- (8.5, 0.5*1.732);
            \draw (9,0) -- (4.5, 4.5*1.732);
            \draw [ultra thick, draw=black, fill=violet, fill opacity=0.5]
            (0,0) -- (3,3*1.732) -- (4,3*1.732) -- (2.5,1.5*1.732) -- (7.5, 1.5*1.732) -- (6,0) -- cycle;
        \end{tikzpicture}
        \caption{An illustration of the bijection for the flag (3, 6, 6, 6, 6, 6, 9, 9, 9).}
        \label{fig:enter-label}
    \end{figure}
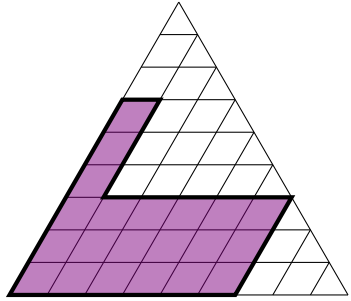
    The proposition thus follows from Proposition~\ref{tflags_and_permutations}.
\end{proof}
\section*{Acknowledgements}
I would like to thank Sankaran Viswanath for his valuable insights, suggestions and corrections. This work also benefited from helpful discussions with K.N. Raghavan and Siddheswar Kundu.
\printbibliography
\end{document}